\documentclass[12 pt]{article}

\usepackage{amsmath}
\usepackage{amssymb}
\usepackage[english]{babel}
\usepackage{amsthm}
\usepackage[latin1]{inputenc}
\usepackage{graphicx}
\usepackage{subfigure}
\usepackage{makeidx}
\usepackage{verbatim}
\usepackage{bbm}

\newtheorem{teo}{Theorem}
\newtheorem{lemma}{Lemma}
\newtheorem{prop}{Proposition}
\newtheorem{defin}{Definition}
\newtheorem{cor}{Corollary}
\newtheorem{remark}{Remark}

\DeclareMathOperator{\Tr}{Tr}
\DeclareMathOperator{\graph}{graph}
\begin{document}

\title{On unbounded motions in a real analytic bouncing ball problem}
\author{Stefano Mar\`o\thanks{
		Department of Mathematics, Faculty of Science, University of Oviedo,
		Oviedo, Spain. E-mail: 
		\texttt{marostefano@uniovi.es}}}
\date{}
\maketitle

\begin{abstract}
We consider the model of a ball elastically bouncing on a racket moving in the vertical direction according to a given periodic function $f(t)$. The gravity force is acting on the ball. We prove that if the function $f(t)$ belongs to a class of trigonometric polynomials of degree $2$ then there exists a one dimensional continuum of initial conditions for which the velocity of the ball tends to infinity. Our result improves a previous one by Pustyl'nikov and gives a new upper bound to the applicability of KAM theory to this model. 
  \end{abstract}

\section{Introduction}
The vertical dynamics of a free falling ball on a moving racket is considered. The racket is supposed to
move in the vertical direction according to a periodic function $f(t)$ and
the ball is reflected according to the law of elastic bouncing when hitting the
racket. The only force acting on the ball is the gravity, with acceleration $g$. Moreover, the mass of the racket is assumed to be large with respect to the mass of the
ball so that the impacts do not affect the motion of the racket.

\noindent This model has inspired many authors as it represents a simple mechanical model exhibiting complex dynamics; see for example \cite{holmes,kunzeortega2,maro3,maro2,maro5,pust,qiantorres,ruiztorres} where results on periodic or quasiperiodic motions are proved together with, in some case, topological chaos. Moreover, for some $f$ presenting some singularities it is possible to study statistical and ergodic properties \cite{dolgo,studolgo}. 

In this paper we are concerned with the existence of unbounded motions, supposing $f$ real analytic. We understand that a motion is unbounded if the velocity of the ball tends to infinity. \\
The first result in this direction is due to Pustyl'nikov assuming that $2\dot{f}(t_0)=g$ for some $t_0$. We remark that with the same proof, it is possible to weaken the condition to 
\begin{equation}\label{pustopt}
\max\dot{f}\geq \frac{g}{4}.
\end{equation}       

\noindent On the other hand, an application of KAM theory implies that there exists $\epsilon>0$ (depending on $g$) such that if the  norm of $\dot{f}$ is smaller than $\epsilon$ then all motions are bounded \cite{siegelmoser}. It is then natural to ask, in a broad sense, what is the optimal condition on $\dot{f}$ for which all motions are bounded. In this respect, condition \eqref{pustopt} represents an upper bound to the norm of $\dot{f}$. 

 \noindent The question is in its generality hard to solve. Actually, it has already attracted a lot of attention for the case of the standard map \cite{ll,mackay}, being the problem still open.   

The purpose of this paper is to improve the upper bound given by \eqref{pustopt}, more precisely, we give an explicit example of a trigonometric polynomial $p(t)$ with

\begin{equation}
	\max\dot{p}< \frac{g}{4} 
\end{equation}  
admitting unbounded bouncing motions.

 We also stress that regularity plays a role, being KAM theory not applicable with too low regularity \cite{herman1, herman2}. Hence our example gives a new upper bound in the analytic class. On this line, we also cite the result in \cite{maro4} where functions with arbitrary small first derivative are constructed admitting unbounded bouncing motions. See also \cite{ma_xu} for a $C^{1+\alpha}$ extension of this result. In any case, these functions are not $C^2$ and it is not clear if a smoothing procedure can be applied. 

The idea behind Pustyl'nikov's result is to construct a motion such that the time difference between two consecutive bounces is a positive integer. Hence, since $f$ is one periodic, the velocity increases of a fixed quantity at every bounce. Moreover, condition \eqref{pustopt} is optimal to get such motions. Hence, we will construct the trigonometric polynomial $p(t)$ such that in the corresponding bouncing motion the increase of the velocity occurs every $N$ bounces, for some $N\geq 2$. This idea was already used in \cite{maro4} and we will show that it can be used to construct also real analytic examples.

Finally, we stress that the trigonometric polynomial $p(t)$ admits a continuum of unbounded orbits and moreover belongs to a family of trigonometric polynomial $p_s(t)$, parametrized by $s$ in a real open interval $I$, with
\begin{equation}
	\max\dot{p}_s< \frac{g}{4} \quad\mbox{for all } s\in I, 
\end{equation}  
each of them admitting a continuum of unbounded motions. 

The existence of a continuum of unbounded motions was already presented in the results of Pustyl'nikov. Here we extend his idea, which we recall here. The bouncing motions correspond to (positive) orbits of a symplectic map of the cylinder $P_f$. The map $P_f$ shares some orbits with a map of the torus and the unbounded orbit constructed by Pustyl'nikov can be seen as a fixed point of this latter map. A condition on $\ddot{f}$ guarantees that this fixed point is hyperbolic and the stable manifold exists. Coming back to the original map, the stable manifold is preserved and allows to construct the desired continuum of initial conditions. We stress that the linearization of $P_f$ along the unbounded orbit is a non autonomous map, so that the just cited result of preservation of the invariant manifold is delicate.

In our case, the unbounded orbit corresponds to a $N$-cycle of the approximating map, and we will show that it is hyperbolic due to the properties of the family $p_s$. Hence we will find the continuum of unbounded orbits showing that we can adapt Pustyl'nikov idea to the $N$-th iterate of the map $P_p$.
         
The paper is organized as follows. In Section \ref{sec:state} we state the problem, recall Pustyl'nikov's result and state our main result. In Section \ref{sec:unb} we discuss the optimality of condition \eqref{pustopt} and construct the family $p_s$ proving that the corresponding map admits one unbounded orbit. Finally, in Section \ref{sec:contunb} we show how to extend the unbounded orbit found in the previous Section to a continuum. In the same Section we also give a detailed discussion of the result on the persistence of the stable manifold under non autonomous perturbations.

\section{Statement of the problem and main result}\label{sec:state}

We consider the vertical motions of a ball bouncing elastically on a vertically moving racket. The only force acting on the ball is the gravity, with acceleration $g>0$ and we assume that the impacts do not affect the motion of the racket that is supposed to move like a $1$-periodic real analytic function $f$. In an inertial frame and choosing as coordinates the time of impact $t$ and the velocity $v$ just after the impact we get that the motion is described by the following map (see \cite{maro5})    
\begin{equation}\label{unb}
  P_f:
  \left\{
	\begin{array}{ll}
		t_1=t_{0}+\frac{2}{g}v_{0}-\frac{2}{g}f[t_1,t_{0}]
		\\
		v_1=v_{0}+2\dot{f}(t_1)-2f[t_1,t_{0}]
	\end{array}
        \right.
\end{equation}
where 
$$
f[t_1,t_{0}]=\frac{f(t_1)-f(t_0)}{t_1-t_0}.
$$
 This is also the formulation considered by Pustil'nikov in \cite{pust}. Another approach based on differential equations was considered by Kunze and Ortega \cite{kunzeortega2} and leads to a map that is equivalent to (\ref{unb}). The map is implicit and turns out to be an embedding for $v>\bar{v}$ for some $\bar{v}$ sufficiently large (see \cite{maro5}). Moreover, by the periodicity of the function $f$, the coordinate $t$ can be seen as an angle. Hence the map $P_f$ is defined on the half cylinder $\mathbb{T}\times (\bar{v},+\infty)$, where $\mathbb{T}=\mathbb{R/\mathbb{Z}}$.
 
 A bouncing motion is a positive orbit of $P_f$ i.e. a sequence $(t_n,v_n)_{n\geq 0}$ such that $(t_n,v_n)=P_f^n(t_0,v_0)$ for every $n\geq 0$. We are interested in the existence of unbounded bouncing motions, where
 \begin{defin}
 	A bouncing motion $(t_n,v_n)_{n\geq 0}$ is said unbounded if
\[
\lim_{n\to +\infty}v_n =+\infty.
\] 
\end{defin}

 The first result on the existence of an unbounded bouncing motion is due to Pustyl'nikov
 \begin{teo}[\cite{pust}]\label{pusteo}
	Suppose that $2\dot{f}(t_0^*)=g$. Then there exists a positive integer $m$ such that choosing $v_0^*=mg/2$, the bouncing motion given by $P_f^n(t^*_0,v_0^*)$ satisfies for $n\geq 0$,
\begin{equation}
	t_{n+1}^*-t_n^*=m+2n, \qquad v_{n+1}^*-v_n^*=g.
\end{equation}
If in addition $\ddot{f}(t_0^*)>0$ or $\ddot{f}(t_0^*)<-g$ then there exists a one-dimensional continuum  $\Gamma\subset\mathbb{T}\times\mathbb{R}$ of initial data containing $(t_0^*,v_0^*)$ such that if $(t,v)\in\Gamma$ then bouncing motion given by $P_f^n(t,v)$ is unbounded and
\[
\lim_{n\to +\infty}\frac{v_n}{n}=g.
\]  
 \end{teo} 
We remark that following Pustyl'nikov proof the condition on $\dot{f}$ can be weakened to $4\dot{f}(t_0^*)=g$ being this condition optimal to get orbits with increasing velocity at every bounce. We prove these facts in Proposition \ref{pustrev} and Remark \ref{pustrem}.   
 
 The main result of this paper is to prove that there exists a continuum of unbounded bouncing motions for a class of functions not satisfying the optimal condition of Pustyl'nikov. More precisely we prove the following
 
 \begin{teo}\label{main}
 	There exists a $1$-periodic trigonometric polynomial $p(t)$ of degree $2$ such that
 	\[
 	\max \dot{p} <\frac{g}{4} 
 	\] 
 	and the map $P_p$ admits a one dimensional continuum of unbounded motions.
 \end{teo}
 It will come from the proof that the expression of $p(t)$ is explicit. Moreover, $p(t)$ belongs to a family of trigonometric polynomials with the same properties: 
 
 \begin{cor}\label{main2}
 	There exist a real open interval I and a one parameter family $p_s(t), s\in I$ of $1$-periodic trigonometric polynomials of degree $2$ such that
 	\begin{itemize}
 		\item[(i)] $p(t)=p_{\bar{s}}(t)$ for some $\bar{s}\in I$,
 		\item[(ii)] $ \max \dot{p}_s < g/4$ for all $s\in I$,
 		\item[(iii)] for every $s\in I$ the corresponding maps $P_{p_s}$ admit a one dimensional continuum $S_s$ such that each bouncing motion starting on $S_s$ is unbounded.
 	\end{itemize}

 \end{cor}
 
\section{Existence of one unbounded motion}\label{sec:unb}

In this section we are going to construct a motion of the racket in the form of a trigonometric polynomial $p(t)$ not satisfying condition \eqref{pustopt}, for which there exits an unbounded bouncing motion.\\
The main idea behind our result is the following. If $(t_n^*,v_n^*)_{n\geq 0}$ is a positive orbit of $P_f$ satisfying 
\begin{equation}\label{condug}
f(t_n^*)=f(t_0^*)\quad \mbox{for every $n\geq 0$}
\end{equation}
then $f[t^*_{n+1},t^*_{n}]=0$ for every $n\geq 0$ and $(t_n^*,v_n^*)_{n\geq 0}$ becomes a positive orbit for the generalized standard map
\begin{equation}\label{standard}
  GS:
  \left\{
\begin{array}{ll}
t_1=t_{0}+\frac{2}{g}v_{0}
\\
v_1=v_{0}+2\dot{f}(t_1)
\end{array}
\right.
\end{equation}
Conversely, if $(t_n^*,v_n^*)_{n\geq 0}$ is a positive orbit of $GS$ with $v_n>\bar{v}$ for every $n$ and satisfying condition (\ref{condug}) then it is also a positive orbit for $P_f$. The idea of Pustyl'nikov was to construct a positive orbit $(t_n^*,v_n^*)_{n\geq 0}$ of $GS$ such that for every $n$, $t_{n+1}^*-t_n^*\in\mathbb{N}\setminus\{0\}$ and
$v_{n+1}^*-v_n^*=g$. In this way, he got an unbounded orbit if $\max \dot{f}\geq g/2$ as stated in Theorem \ref{pusteo}. With the same idea the result can be improved a little. Here we report a complete proof.
\begin{prop}\label{pustrev}
	Suppose that
	\begin{equation}\label{conpus}
	\max \dot{f}\geq g/4.
	\end{equation}
	Then there exist $t_0^*$ and a positive integer $m$ such that the bouncing motion with initial condition $(t^*_0,v_0^*=mg/2)$ satisfies for every $n\geq 0$
	\begin{equation}\label{pustorbit}
	t_{n+1}^*-t_n^*=n+m, \qquad v_{n+1}^*-v_n^*=g/2.
	\end{equation}
\end{prop} 
\begin{proof}
By hypothesis, there exists a point $t_0^*$ such that  $\dot{f}(t_0^*)= g/4$. Choose $v_0^*=mg/2$ with an integer $m$ such that $v_0>\bar{v}$. For $n\geq 0$, we consider $(t_n^*,v_n^*)=GS^n(t_0^*,v_0^*)$ and we prove by induction that it satisfies \eqref{pustorbit}, so that \eqref{condug} is satisfied and $(t_n^*,v_n^*)$ is also an orbit of $P$ with the desired properties. The case $n=0$ is obvious. Moreover, using the inductive hypothesis
\[
t_{n+2}^*-t_{n+1}^*=\frac{2}{g}v^*_{n+1}=\frac{2}{g}v^*_n+1=\frac{2}{g}(v^*_0+n\frac{g}{2})+1=(n+1)+m
\]          	
and
\[
 v_{n+2}^*-v_{n+1}^*=2\dot{f}(t^*_{n+2})=2\dot{f}(t^*_{n+1}+(n+1)+m)=2\dot{f}(t^*_{0})=g/2.
\]
\end{proof}

\begin{remark}\label{pustrem}
	Pustyl'nikov orbit is such that the velocity increases at every iterate of the map and the bouncing time are equal modulo $1$. We note that if $\max \dot{f}<g/4$ then this is not possible. Actually, suppose that $v_{n+1}^*-v_n^*>0$ and $t_{n+1}^*-t_n^*=\sigma_n\in\mathbb{N}$ for every $n\geq 0$. Then $f[t^*_{n+1},t^*_{n}]=0$ and $t_{n+1}^*-t_n^*=\frac{2}{g}v^*_n=\sigma_n$ with $\sigma_{n+1}>\sigma_n$. This implies that $v^*_n=\frac{g}{2}\sigma_n$ and
	\[
	\frac{g}{2}(\sigma_{n+1}-\sigma_{n})=v_{n+1}^*-v_n^*=2\dot{f}(t_{n+1}^*)<\frac{g}{2}
	\]
	that is a contradiction.      
\end{remark}

 In view of the previous Remark, we search for conditions to construct unbounded motions such that the velocity increases every $N$ iterates. This is achieved in \cite{maro4} with the following result.   

\begin{prop}[\cite{maro4}]\label{stand2}
	Consider a function $f\in C^1(\mathbb{R}/\mathbb{Z})$ and a sequence $(t_n^*)_{n\in\mathbb{N}}$. Suppose that there exist three positive integers $N, W, V$ such that
	
	\begin{enumerate}
		\item $t_N^*-t_0^*=W$,
		\item $\frac{4}{g}\dot{f}(t_0^*)+(t_N^*-t_{N-1}^*)-(t_1^*-t_0^*)=V$,
		\item $f(t^*_0)=f(t^*_1)=\dots =f(t^*_{N-1})$,
		\item $\dot{f}(t_k^*)=\frac{g}{4}(t_{k+1}^*-2t_k^*+t_{k-1}^*)$ for $1\leq k\leq N-1$.
	\end{enumerate}

	Then if we define $v^*_{n+1}=v^*_{n}+2\dot{f}(t^*_{n+1})$ and $v^*_0=\frac{g(t^*_1-t^*_0)}{2}$ we have that there exists an orbit $(\bar{t}_n^*,\bar{v}_n^*)_{n\in\mathbb{N}}$ of $P_f$ such that $(\bar{t}_n^*,\bar{v}_n^*)=(t_n^*,v_n^*)$ for $0\leq n\leq N$ and  
	$$
	\bar{t}_{n+N}^*=\bar{t}_n^*+\sigma_n,\quad \sigma_n\in\mathbb{N}
	$$
	$$
	\bar{v}^*_{n+N}=\bar{v}_n^*+\frac{g}{2}V.
	$$
	Moreover, there exists $T>0$ such that if $t_1^*-t_0^*>T$ then $v^*_n>\bar{v}$ for every $n\geq 0$. 
\end{prop}

\begin{remark}
For $N=1$ all the conditions are satisfied if
\begin{equation*}
\frac{4}{g}\dot{f}(t_0^*)\in\mathbb{N}\setminus\{0\}.
\end{equation*}
In this case we can select $t_1^*$ such that $t_1^*-t_0^*$ is a sufficiently large integer. Hence we get the result in Proposition \ref{pustrev}. 
\end{remark}
Next we show that Proposition \ref{stand2} can be used to find a real analytic function $p$ that does not satisfy Pustyl'nikov condition \eqref{conpus} and admits unbounded orbits. More precisely:

\begin{prop}\label{oneparfam}
	There exist a real open interval $I$ and a one parameter family $p_s(t), s\in I$ of $1$-periodic trigonometric polynomials of degree $2$ such that
	\[
	\max \dot{p}_s <\frac{g}{4} \quad \forall s\in I
	\]
	and the corresponding maps $P_{p_s}$ admit an unbounded motion as in Proposition \eqref{stand2} independent on $s$.
\end{prop}
\begin{proof}
	Consider a generic trigonometric polynomial of degree $2$
	\begin{equation}\label{generaltp}
	p(t) = a_1 \sin(2\pi t)+b_1 \cos(2\pi t) + a_2 \sin(4\pi t) + b_2 \cos(4\pi t). 
	\end{equation}
	We search for conditions on the real parameters $a_1,b_1,a_2,b_2$ such that, properties 1. 2. 3. 4. in Proposition \ref{stand2} are satisfied for $N=2$ and some $t_0^*<t_1^*$. These can be rewritten as 

	\begin{enumerate}
		\item[1.] $t_2^*-t_0^*=W$
		\item[2'.] $2(t_1^*-t_0^*)+\frac{4}{g} \dot{p}(t_1^*)=W$,
		\item[3.] $p(t^*_0)=p(t^*_1)$,
		\item[4'.] $\frac{4}{g}(\dot{p}(t_0^*)+\dot{p}(t_1^*))=V$.
	\end{enumerate}
   	   	Actually, conditions 1. and 2'. imply condition 4., and using condition 4'. we get also condition 2. Hence, let us search for conditions on $a_1,b_1,a_2,b_2$ in order to satisfy properties 1. 2'. 3. 4'. Let us choose $t_0^*=0$, $t_1^*=5/12+k$ where $k>T$ is a positive integer and $V=W-2k=1$. With these values fix $t_2^*$ to comply with condition 1. Hence conditions 2'.,3.,4'. represent a linear system of 3 equations in the unknowns $a_1,b_1,a_2,b_2$ with parametric solution
	\begin{equation}\label{coeff}
		\begin{split}
			&a_1(s)=gs, \quad b_1(s)=g\left((2-\sqrt{3})s+\frac{4\sqrt{3}-7}{4\pi}\right) \\
			&a_2(s)=g\left(\frac{5}{96\pi}-\frac{s}{2}\right), \quad b_2(s)=g\left(\frac{\sqrt{3}}{2}s+\frac{48-29\sqrt{3}}{96\pi} \right).
		\end{split}
	\end{equation}

Substituting these values in \eqref{generaltp} we get a one parameter family  $p_s(t)$ with $s\in \mathbb{R}$
satisfying 1., 2'., 3., 4' for every $s$ with the same values of $t_0^*,t_1^*$. 
Hence, for every $s\neq 0$ we can apply Proposition \ref{stand2} to $P_{p_s}$ and find the same unbounded motion. \\ 
Finally, to estimate the maximum of $\dot{p}_s$, using 
$a\sin x + b\cos x\leq \sqrt{a^2+b^2}$, we have that
for every $s$, 
\begin{equation}\label{maxps}
\begin{split}
	\max \dot{p}_s &\leq 2\pi\sqrt{a_1(s)^2+b_1(s)^2}+4\pi\sqrt{a_2(s)^2+b_2(s)^2}=:g\bar{p}(s). 
\end{split}
\end{equation}
Now a direct computation shows that $\bar{p}(0.006)<1/4$ so that by continuity there exists a real open interval $I\ni 0.006$ such that for every $s\in I$, $\max \dot{p}_s<g/4$. Finally, we can also restrict $I$ in such a way that $0\notin I$. In this way $a_1\neq 0$ and $p_s(t)$ has minimal period $1$ for every $s\in I$.   

\end{proof}

\begin{remark}\label{numeric}
	It is possible to have a numerical approximation of the optimal value for $\max\dot{p}_s$, coming from formula \eqref{maxps}. Actually the minimum of $\bar{p}_s$ can be approximated by
	\[
	\bar{s}=0.009569094523943, \qquad \bar{p}_s(\bar{s})=0.211931840664873.
	\]  
	 
\end{remark}

Note that Proposition \ref{oneparfam} gives a proof of Theorem \ref{main} and Corollary \ref{main2}, items \emph{(i),(ii)}, concerning the existence of just one unbounded motion. The proof of the existence of a continuum of unbounded motions is the purpose of the next section.

\section{Existence of a continuum of unbounded motions}\label{sec:contunb}

By now we have an unbounded motion. In this Section we prove that it can be extended to a continuum of unbounded motions. A key role is played by a theorem on non autonomous perturbations of hyperbolic maps.
  
\subsection{Non autonomous perturbations of hyperbolic maps}
 In this subsection we are going to discuss and prove the following result, a discrete version of a classical theorem of differential equations (see \cite[Theorem 4.1 pp 330]{coddingtonlevinson}).
 
\begin{teo}\label{cl}
Consider the difference equation
$$
x_{n+1}=Ax_n+R_n(x_n)
$$
where $A$ is a $m\times m$ matrix such that $k$ eigenvalues have modulus less than $1$ and the remaining $m-k$ have modulus greater than $1$.
Let $\mathcal{U}$ be an open neighbourhood of the origin and suppose that $R_n:\mathcal{U}\subset\mathbb{R}^m\rightarrow\mathbb{R}^m$ is continuous for every $n$ and $R_n(0)=0$. Moreover suppose that for every $\epsilon>0$ there exist $\delta>0$ and $M>0$ such that for every $n>M$ and $u,v\in \mathcal{B}_\delta$
$$
|R_n(u)-R_n(v)|\leq\epsilon |u-v|
$$
with $\mathcal{B}_\delta$ representing the ball centered in $0$ with radius $\delta$.  
Then there exists $n_0$ such that for every $n\geq n_0$ there exists a $k$-dimensional topological manifold $S=S_{n}$ passing through zero such that if $x_{n_0}\in S_{n_0}$ then $x_{n}\in S_{n}$ for every $n\geq n_0$ and $x_n\to 0$ as $n>n_0$ tends to $+\infty$. 
\end{teo}

The rest of the subsection is dedicated to the proof of this result. 
Let us start with some preliminaries. There exists a real nonsingular matrix $P$ such that
\[
PAP^{-1}=B=\left(
\begin{array}{cc}
	B_1 &  0 \\ 
	0 & B_2
\end{array}
\right)
\]
and $B_1$ is a $k\times k$ matrix with all eigenvalues with modulus less than one and $B_2$ is a $(m-k)\times (m-k)$ matrix with all eigenvalues with modulus greater than one.
By the change of variables $y=Px$ our system becomes equivalent to 
\begin{equation}\label{eqy}
	y_{n+1}=By_n+g_n(y_n)
\end{equation}
with $g_n(y_n)=PR_n(P^{-1}y_n)$. Moreover, $g(0)=0$ and given any $\epsilon>0$, there exist $M_\epsilon$ and $\delta_\epsilon$ such that
\begin{equation}\label{lipg}
	|g_n(u)-g_n(v)|\leq\epsilon |u-v|
\end{equation}
for $n>M_\epsilon$ and $u,v\in \mathcal{B}_{\delta_\epsilon}$.  
Now define 
\begin{equation*}
	U_1=
	\left(
	\begin{array}{cc}
		B_1 &  0 \\ 
		0 & 0
	\end{array}
	\right) 
\end{equation*}
and
\begin{equation*}
	U_2=
	\left(
	\begin{array}{cc}
		0 &  0 \\ 
		0 & B_2
	\end{array}
	\right). 
\end{equation*}
Note that $U^n=(U_1+U_2)^n$ is the matrix solution of the linear system $y_{n+1}=By_n$. Moreover, there exist $K>0$, $\alpha_1<1$ and $\alpha_2<1$ such that 
\begin{equation}\label{sti1}
	|U_1^n|\leq K\alpha_1^n,
\end{equation}
\begin{equation}\label{sti2}
	|U_2^{-n}|\leq K\alpha_2^n
\end{equation}
where $| \cdot |$ is the associated matrix norm. With this notation, equation \eqref{eqy} can be rewritten in an integral form
\begin{lemma}\label{lemequiv}
	Suppose that for $a\in\mathbb{R}^m$ and $n_0>0$ there exists a sequence $\{\theta_n(a)\}_{n\geq n_0}$ such that
	\begin{equation}\label{lime}
		\theta_n(a)=U_1^{n-n_0}a  +\sum_{s=n_0}^{n-1} U_1^{n-s-1}g_s(\theta_s(a))
		-\sum_{s=n}^\infty U_2^{n-s-1}g_s(\theta_s(a))
	\end{equation}
for every $n>n_0$ where the last sum is convergent. Then this sequence satisfies the difference equation \eqref{eqy}.
\end{lemma}
\begin{proof}
	Equation \eqref{eqy} can be written as $y_{n+1}=(U_1+U_2)y_n+g_n(y_n)$. Then,
	\begin{align*}
	\theta_{n+1}(a)&=U_1^{n-n_0+1}a  +\sum_{s=n_0}^{n} U_1^{n-s}g_s(\theta_s(a))
	-\sum_{s=n+1}^\infty U_2^{n-s}g_s(\theta_s(a)) \\
	&=U_1\left(U_1^{n-n_0}a  +\sum_{s=n_0}^{n} U_1^{n-s-1}g_s(\theta_s(a))\right)
	-U_2\sum_{s=n+1}^\infty U_2^{n-s-1}g_s(\theta_s(a)) \\
	&=U_1\left(U_1^{n-n_0}a  +\sum_{s=n_0}^{n-1} U_1^{n-s-1}g_s(\theta_s(a))\right)
	-U_2\sum_{s=n}^\infty U_2^{n-s-1}g_s(\theta_s(a)) \\ 
	&\phantom{=}+ U_1U_1^{-1}g_n(\theta_n(a))+U_2U_2^{-1}g_n(\theta_n(a)) \\
	&=(U_1+U_2)\theta_{n}(a)+g_n(\theta_n(a)) 	
			\end{align*}  
where in the last equality we have used the fact that $U_1$ only acts on the first $k$ rows and $U_2$ only on the last $m-k$ and their images are subspaces dependent only on the first $k$ and $m-k$ coordinates respectively.  

\end{proof}
The following lemma is crucial
\begin{lemma}\label{lemcruc}
	There exist $n_0>0, \delta>0$ and a sequence $\{\theta_n(a)\}_{n\geq n_0}$ satisfying the following properties
	\begin{enumerate}
		\item the function $a\mapsto\theta_n(a)$ is continuous for $n\geq n_0$ and $|a|<\delta$,
		\item $\theta_n(a)\to 0$ as $n\to\infty$ uniformly for $|a|<\delta$,
		\item $\{\theta_n(a)\}_{n\geq n_0}$ satisfies \eqref{lime}.  
	\end{enumerate}  
\end{lemma}
We postpone the proof to the end of the subsection and we show how we can conclude. First note the from \eqref{lime} the last $m-k$ components of $a$ do not enter in the solution $\theta_n(a)$ so that we can set them at $0$. Consider now, for $|a|$ sufficiently small the continuous function 
\[
\theta_{n_0}(a)=
\left(
\begin{array}{cc}
	I_k &  0 \\ 
	0 &  0
\end{array}
\right)
a-\sum_{s=n_0}^\infty U_2^{n_0-s-1}g_s(\theta_s(a)),
\]
where $I_k$ represents the identity matrix of order $k$.
Notice that, by the definition of the matrix $U_2$, the first $k$ components of $\theta_{n_0}(a)$ are
\[
(\theta_{n_0}(a))_j = a_j, \qquad j=1,\dots,k
\] 
and the remaining $m-k$ are, 
\[
(\theta_{n_0}(a))_j = (\phi_{n_0}(a))_j, \qquad j=k+1,\dots,m.
\]
where $\phi_{n_0}(a)=-\sum_{s=n_0}^\infty U_2^{n_0-s-1}g_s(\theta_s(a))$ is continuous in $a$.
 Then we can see $\theta_{n_0}(a)$ in the form
$$
\theta_{n_0}(a)=(a_1,\dots,a_k,(\phi_{n_0}(a_1,\dots,a_k))_{k+1},\dots,(\phi_{n_0}(a_1,\dots,a_k))_{m})
$$ 
where $\phi_{n_0}(a_1,\dots,a_k)=-\sum_{s=n_0}^\infty U_2^{n_0-s-1}g_s(\theta_s(a_1,\dots,a_k,0,\dots,0))$ is continuous in $a_1,\dots,a_k$. Hence $S_{n_0}:=\graph\phi_{n_0}$ represents a $k$ dimensional manifold in $\mathbb{R}^m$. Analogously, for $n>n_0$ we consider for $|a|$ small the continuous function
\[
\theta_n(a)=U_1^{n-n_0}a  +\sum_{s=n_0}^{n-1} U_1^{n-s-1}g_s(\theta_s(a))
		-\sum_{s=n}^\infty U_2^{n-s-1}g_s(\theta_s(a))
\]
and note that the components of $\theta_{n}(a)$ are
\[
(\theta_{n}(a))_j =\left\{
\begin{array}{l}
(\varphi_n(a))_j, \qquad j=1,\dots,k \\
(\phi_{n}(a))_j, \qquad j=k+1,\dots,m.
\end{array}
  \right.
\]
where $\varphi_n(a) = U_1^{n-n_0}a  +\sum_{s=n_0}^{n-1} U_1^{n-s-1}g_s(\theta_s(a))$ and $\phi_{n}(a)=-\sum_{s=n}^\infty U_2^{n-s-1}g_s(\theta_s(a))$ are continuous in $a$.  Then, as before, we only consider the first $k$ components of $a$ and define the $k$ dimensional manifold $S_n$ parametrically using the first $k$ components of $\varphi_n(a_1,\dots,a_k)$ and the last $m-k$ components of $\phi_{n}(a_1,\dots,a_k)$. \\
Suppose now that there exists an orbit of \eqref{eqy} such that $y_{n_0}\in S_{n_0}$. Then there exists $a$ such that
\[
y_{n_0} = \theta_{n_0}(a)=a-\sum_{s=n_0}^\infty U_2^{n_0-s-1}g_s(\theta_s(a))
\]  
and $\theta_n(a)$ satisfies \eqref{lime} for $n\geq n_0$. By the uniqueness of the solutions of (\ref{eqy}-\ref{lipg}) and Lemma \ref{lemequiv}, $y_n=\theta_n(a)$ for $n\geq n_0$.   Moreover, by the definition of $S_n$ we have that $y_n\in S_n$ for every $n\geq n_0$ and, by Lemma \ref{lemcruc}, $y_n\to 0$ as $n\to\infty$. Coming back to the $x$ variables we have the thesis. 

To conclude the proof we give the

\begin{proof}[Proof of Lemma \ref{lemcruc}]
With reference to constants introduced in (\ref{lipg},\ref{sti1},\ref{sti2}), let $\max\{\alpha_1,\alpha_2\}<\alpha<1$ and choose $\epsilon$ such that $\epsilon K\left(\frac{1}{1-\alpha_1}+\frac{\alpha_2}{1-\alpha_2}\right)<\frac{1}{2}$. Let $n_0>M_\epsilon$ and $a\in\mathbb{R}^m$ such that $2K|a|<\delta_\epsilon$.

For each $l\geq 0$ consider the sequence $\{\theta_n^{(l)}(a)\}_{n\geq n_0}$ defined by induction on $l$ by
\begin{equation}\label{varcost}
	\left\{
	\begin{aligned}
		&\theta_n^{(0)}(a)=0 \quad \forall n\geq n_0\\
		&\theta_n^{(l+1)}(a)=U_1^{n-n_0}a  +\sum_{s=n_0}^{n-1} U_1^{n-s-1}g_s(\theta_s^{(l)}(a))-\sum_{s=n}^\infty U_2^{n-s-1}g_s(\theta_s^{(l)}(a))\\
	\end{aligned}
	\right.
\end{equation}
Let us prove by induction on $l$ that this sequence is well-defined and for $n\geq n_0$
\begin{equation}\label{lips}
	\begin{split}
		|\theta_n^{(l+1)}(a)-\theta_n^{(l)}(a)|&\leq\frac{K|a|\alpha^{n-n_0}}{2^l}\quad\mbox{and} \\
		|\theta_n^{(l)}(a)|&<\delta_\epsilon
	\end{split}
\end{equation}

The step $l=0$ comes readily from estimate (\ref{sti1}) and the choice of $\alpha$. Notice that the hypothesis $|\theta_n^{(l)}(a)|<\delta_\epsilon$ allows to say that $|\theta_n^{(l+1)}(a)|$ is well defined because the last sum is dominated by a convergent geometric series. Moreover, by the inductive hypothesis we have that for $h=0,\dots,l$,
\begin{equation}\label{stimt}
	\begin{split}
		|\theta_n^{(h+1)}(a)|&\leq |\theta_n^{(h+1)}(a)-\theta_n^{(h)}(a)|+|\theta_n^{(h)}(a)-\theta_n^{(h-1)}(a)|+\dots +|\theta_n^{(0)}(a)|\\
		&\leq \sum_{i=0}^h \frac{K|a|\alpha^{n-n_0}}{2^i}\leq K|a|\sum_{i=0}^{+\infty}\frac{1}{2^i}=2K|a|<\delta_\epsilon
	\end{split}
\end{equation}
remembering the choice of $a$.
So also $|\theta^{(l+2)}_n|$ is well defined and we have
\begin{equation*}
	\begin{split}
		|\theta_n^{(l+2)}(a)-\theta_n^{(l+1)}(a)|\leq & \sum_{s=n_0}^{n-1} K\alpha_1^{n-s-1}|g_s(\theta_s^{(l+1)}(a))-g_s(\theta_s^{(l)}(a))|\\
		& +\sum_{s=n}^\infty K\alpha_2^{s-n+1}|g_s(\theta_s^{(l+1)}(a))-g_s(\theta_s^{(l)}(a))|.     
	\end{split}
\end{equation*}
So, by (\ref{stimt}) we can use (\ref{lipg}) and the inductive hypothesis to get
\begin{equation*}
	\begin{split}
          |\theta_n^{(l+2)}(a)-\theta_n^{(l+1)}(a)|&\leq \frac{K|a|\alpha^{n-n_0}}{2^l}\epsilon K\left(\frac{\alpha_1^{n-n_0}-1}{\alpha_1-1}+\frac{\alpha_2}{1-\alpha_2}\right)\\
          &\leq \frac{K|a|\alpha^{n-n_0}}{2^l}\epsilon K\left(\frac{1}{1-\alpha_1}+\frac{\alpha_2}{1-\alpha_2}\right)
\end{split}
\end{equation*}
that allows us to conclude remembering the choice of $\epsilon$.\\
Let us now define $\Delta_n^{(l)}(a)=|\theta_n^{(l+1)}(a)-\theta_n^{(l)}(a)|$ and note that by (\ref{lips}) and the choice of $|a|$ we have $\Delta_n^{(l)}(a)<\delta_\epsilon\frac{\alpha^{n-n_0}}{2^{l+1}}$. Hence the series $\sum_{l=0}^\infty\Delta_n^{(l)}(a)$ is uniformly and absolutely convergent being dominated by the series $\sum_{l=0}^\infty\delta_\epsilon2^{-(l+1)}$. Consequently, the partial sum
\[
\sum_{l=0}^{\ell-1}(\theta_n^{(l+1)}(a)-\theta_n^{(l)}(a))=\theta_n^{(\ell)}(a)
\]
tends uniformly to a limit $|\theta_n (a)|$ and, by the Weierstrass Test the function $a\mapsto\theta_n(a)$ is continuous for $n\geq n_0$ and $|a|<\frac{\delta_\epsilon}{2K}$. Moreover,
\begin{equation}\label{tendo}
  |\theta_n (a)|=|\theta_n (a)-\theta_n^{(0)}(a)|\leq\sum_{l=0}^\infty\Delta_n^{(l)}(a)< \delta_\epsilon\alpha^{n-n_0}\sum_{l=0}^\infty\frac{1}{2^{l+1}}\leq  \delta_\epsilon \alpha^{n-n_0}.
\end{equation}
Now we want to pass to the limit in (\ref{varcost}). Notice that in order to pass the limit into the last sum we have to use the dominated convergent theorem noticing that for every $s\geq n$
$$
U_2^{n-s-1}g_s(\theta_s^{(l)}(a))\rightarrow U_2^{n-s-1}g_s(\theta_s(a))\quad\mbox{uniformly}
$$ 
and 
$$
|U_2^{n-s-1}g_s(\theta_s^{(l)}(a))|\leq C\alpha^n.
$$
So we are lead to the equation
\begin{equation*}
	\theta_{n}(a)=U_1(n-n_0)a  +\sum_{s=n_0}^{n-1} U_1(n-s-1)g_s(\theta_s(a))
	-\sum_{s=n}^\infty U_2(n-s-1)g_s(\theta_s(a))
\end{equation*}
that is \eqref{lime}. Finally, by (\ref{tendo}) the just defined sequence $\{\theta_n(a)\}_{n\geq n_0}$ tends to $0$ uniformly as $n\to\infty$.

\end{proof}

\subsection{Application to the bouncing ball problem} 
Consider the map $P_{f}(t_n,v_n)$ defined by
\begin{equation}\label{orig}
	\left\{
	\begin{aligned}
		t_{n+1}&=t_{n}+\frac{2}{g}v_{n}-\frac{2}{g}f[t_{n+1},t_{n}]
		\\
		v_{n+1}&=v_{n}+2\dot{f}(t_{n+1})-2f[t_{n+1},t_{n}]
	\end{aligned}
\right.
\end{equation}
and suppose that there exists an unbounded positive orbit $(t_n^*,v_n^*)_{n\geq 0}$ such that for some positive integers $N,V$ and for every $n\geq 0$, 
\begin{equation}\label{unborb}
t^*_{n+N}-t^*_n\in\mathbb{N}\setminus\{0\}, \quad v^*_{n+N}-v_n^*=\frac{g}{2}V, \quad f[t^*_{n+1},t^*_n]=0.
\end{equation}
Note that this is guaranteed if the function $f$ belongs to the family $p_s$ defined in Proposition \ref{oneparfam}.
Let us consider the change of variables 
\begin{equation*}
	\left\{
	\begin{aligned}
		\tau_n&=t_n-t_n^* 
		\\
		\nu_n&=v_n-v_n^*
	\end{aligned}
\right.   
\end{equation*}
Then we have the following expansion of the associated map $P_f(\tau_n,\nu_n)$.
\begin{prop}\label{proplin}
The map $P_f(\tau_n,\nu_n)$ can be written, denoting $x_n=(\tau_n,\nu_n)$ as
\[
x_{n+1}=A_{n+1}x_n+\Omega_n(x_n)   
\]
where
\[
A_n=\left(\begin{array}{cc}
	1 & \frac{2}{g}  \\ 
	2\ddot{f}(t_{n}^*) & 1+\frac{4}{g}\ddot{f}(t_{n}^*)
\end{array}
\right) 
\] 
 and $\Omega_n$ satisfies $\Omega_n(0)=0$ and for every $\epsilon>0$ there exist $\delta>0$ and $M>0$ such that for every $n>M$ and $u,v\in \mathcal{B}_\delta$
 $$
 |\Omega_n(u)-\Omega_n(v)|\leq\epsilon |u-v|
 $$
 with $\mathcal{B}_\delta$ representing the ball centered in $0$ with radius $\delta$. 	
\end{prop}
 
In order to prove Proposition \ref{proplin} let us prove the following technical lemma
\begin{lemma}\label{tec}
The equation 
$$
t_1=t+\frac{2}{g}v-\frac{2}{g}f[t,t_1]
$$
has a unique solution $t_1=T(t,v)\geq t+1$ for large $v$. Moreover $T$ is differentiable on $\{(t,v): v>\bar{v}\}$ and
\begin{equation*}
\begin{split}
T&=t+\frac{2}{g}v+O(\frac{1}{v})\\
\frac{\partial T}{\partial t}&=1+O(\frac{1}{v}),\qquad\frac{\partial T}{\partial v}=\frac{2}{g}+O(\frac{1}{v})
\end{split}
\end{equation*}
as $v\to +\infty$.
\end{lemma}
\begin{proof}
The first part has been proved in \cite[Lemma 1]{maro3}, so let us prove the validity of the asymptotic expansions. Let $\Delta=\Delta(t,v)=T-t$ so that, from the definition of $T$, we have
$$
\Delta=\frac{2}{g}v-\frac{2}{g}\frac{f(T)-f(t)}{\Delta}
$$
that is equivalent to the quadratic equation 
$$
\Delta^2-\frac{2}{g}\Delta v+\frac{2}{g}(f(T)-f(t))=0
$$
with roots
$$
\Delta_\pm = \frac{1}{g}v\pm\sqrt{\frac{1}{g^2}v^2-\frac{2}{g}(f(T)-f(t))}.
$$
Then $\Delta_+\to +\infty$ and $\Delta_-\to 0$ as $v\to +\infty$. Since we know that $\Delta\geq 1$, then it must coincide with the positive branch for large $v$. So
$$
\Delta = \frac{1}{g}v +\sqrt{\frac{1}{g^2}v^2-\frac{2}{g}(f(T)-f(t))}= \frac{2}{g}v+O(\frac{1}{v})
$$
that is the expansion for $T$.

We can obtain the expansions for the partial derivatives differentiating the formula
$$
T=t+\frac{2}{g}v-\frac{2}{g}\frac{f(T)-f(t)}{T-t}
$$
and remembering the previous expansion of $T$.
\end{proof}

 Now we are ready for the
\begin{proof}[Proof of Proposition \ref{proplin}]
	Using the variables $(\tau_n,\nu_n)$ the map $P_f(\tau_n,\nu_n)$ takes the from
	\begin{equation}\label{modifi}
          \left\{
		\begin{array}{ll}
			\tau_{n+1}=\tau_{n}+\frac{2}{g}\nu_{n}-\frac{2}{g}\lambda_n(\tau_n,\nu_n)
			\\
			\nu_{n+1}=\nu_{n}+2\phi_n(\tau_n,\nu_n)-2\dot{f}(t_{n+1}^*)-2\lambda_n(\tau_n,\nu_n)
		\end{array}
                \right.
	\end{equation}
	where, recalling that $f[t^*_{n+1},t^*_n]=0$,
        $$
	\lambda_n(\tau,\nu)=f[T(\tau + t^*_{n},\nu + v_n^*),\tau+t_n^*]
	$$
and
	$$
	\phi_n(\tau,\nu)=\dot{f}(T(\tau + t^*_{n},\nu + v_n^*)).
	$$
        
	Rewriting (\ref{modifi}) as a perturbation of the linear map induced by the matrix $A_n$, we get 
	\begin{equation*}
          \left\{
		\begin{array}{ll}
			\tau_{n+1}=\tau_{n}+\frac{2}{g}\nu_{n}-\frac{2}{g}\lambda_n(\tau_n,\nu_n)
			\\
			\nu_{n+1}=\nu_{n}+2\ddot{f}(t_{n+1}^*)(\tau_n + \frac{2}{g}\nu_n)+2r_n(\tau_n,\nu_n)-2\lambda_n(\tau_n,\nu_n)
		\end{array}
                \right.
	\end{equation*}
        where 
	$$
	r_n(\tau,\nu)=\phi_n(\tau,\nu)-\dot{f}(t_{n+1}^*)-\ddot{f}(t_{n+1}^*)(\tau + \frac{2}{g}\nu).
	$$    
	By the regularity of $f$, the functions $\lambda_n$ and $r_n$ are well defined and $C^1$ in a common neighbourhood $\mathcal{U}$ of the origin. Moreover, as $T(t_n^*,v_n^*)=t_{n+1}^*$, we have that $\lambda_n(0,0)=r_n(0,0)=0$. From the asymptotic estimates of Lemma \ref{tec} we deduce that 
	\begin{equation}\label{frt}
		T(\tau + t^*_{n},\nu + v_n^*)-t_{n+1}^*-\tau-\frac{2}{g}\nu=O(\frac{1}{v_n^*})\quad\mbox{ as }n\to +\infty 
	\end{equation}
	uniformly in $\mathcal{U}$. As a consequence, remembering that $t^*_{n+1}-t^*_n=\frac{2}{g}v_n^*\to\infty$ as $n\to\infty$, we have
	$$
	T(\tau + t^*_{n},\nu + v_n^*)-t_{n}^*\to\ +\infty\quad \mbox{ as }n\to +\infty
	$$
	uniformly in $\mathcal{U}$. From this it is easy to verify that
	$$
	||\nabla\lambda_n||_{L^\infty(\mathcal{U})}\to 0\quad\mbox{ as }n\to +\infty. 
	$$
	Now a simple application of the mean value theorem gives, for every $\epsilon$, the existence of $M>0$ such that
	\begin{equation}\label{lamb}
		|\lambda_n(x)-\lambda_n(y)|<\epsilon |x-y|
	\end{equation}
	for every $x,y \in \mathcal{U}$ and $n>M$. To estimate $\nabla r_n$ we have to proceed with more care. We will consider $\frac{\partial r_n}{\partial\tau}$, being the other case similar. We have
	\begin{equation*}
		\begin{split}
			\frac{\partial r_n}{\partial\tau}=&\ddot{f}(T(\tau + t^*_{n},\nu + v_n^*))\frac{\partial T}{\partial\tau}(\tau + t^*_{n},\nu + v_n^*) -\ddot{f}(t_{n+1}^*) \\
			& =\ddot{f}(T(\tau + t^*_{n},\nu + v_n^*)-\ddot{f}(t_{n+1}^*)+O(\frac{1}{v_n^*})\quad \mbox{ as }n\to +\infty
		\end{split}
	\end{equation*}
	where we have used the estimate of Lemma \ref{tec}. Now, the mean value theorem implies that
	\begin{equation*}
		\begin{split}
			|\ddot{f}(T(\tau + t^*_{n},\nu + v_n^*)-\ddot{f}(t_{n+1}^*)|&\leq ||\dddot{f}||_\infty |T-t_{n+1}^*|\\
			&=||\dddot{f}||_\infty |\tau+\frac{2}{g}\nu|+O(\frac{1}{v_n^*})
		\end{split}
	\end{equation*}
	where we have used (\ref{frt}). A similar estimates holds for $\frac{\partial r_n}{\partial\nu}$, so, for every $\epsilon$ we can eventually restrict the neighbourhood $\mathcal{U}$ such that
	$$
	||\nabla r_n||_{L^\infty(\mathcal{U})}<\epsilon \quad\mbox{ as }n\to +\infty. 
	$$ 
	As before, using the mean value theorem we can find for every $\epsilon$, the existence of $M>0$ such that
	\begin{equation*}
		|r_n(x)-r_n(y)|<\epsilon |x-y|
	\end{equation*}
	for every $x,y \in \mathcal{U}$ and $n>M$. This last estimate, together with (\ref{lamb}), allows to conclude letting $x_n=(\tau_n,\nu_n)$ and $\Omega_n(x_n)=(-\frac{2}{g}\lambda_n(x_n),2r_n(x_n)-2\lambda_n(x_n))$.
\end{proof} 
 
 Note that we cannot apply Theorem \ref{cl} to the map $P_f(\tau_n,\nu_n)$ since the matrix $A_n$ coming from Proposition \ref{proplin} is not constant. However, remembering the properties of the unbounded orbit $(t_n^*,v_n^*)$ we can consider the $N$-th iterate $P^N_f(\tau,\nu)$. More precisely,

 \begin{prop}\label{linN}  
 The map $P^N_f(\tau_0,\nu_0)$ can be written, denoting $y_m=(\tau_{mN},\nu_{mN})$ as
 \[
 y_{m+1}=Ay_m+R_m(y_m)   
 \]
 where
 \[
 A=A_{0}A_{N-1}\dots A_{1} 
 \]
 is independent on $m$ 
 and $R_m$ satisfies $R_m(0)=0$ and for every $\epsilon>0$ there exist $\delta>0$ and $M>0$ such that for every $n>M$ and $u,v\in \mathcal{B}_\delta$
 $$
 |R_m(u)-R_m(v)|\leq\epsilon |u-v|
 $$
 with $\mathcal{B}_\delta$ representing the ball centered in $0$ with radius $\delta$. 	
\end{prop}
\begin{proof}
	From a direct computation from Proposition \ref{proplin}
	\begin{align*}
	y_{m+1}&=P^N_f(x_{mN})= A_{(m+1)N}A_{mN+N-1} \cdots A_{mN+1}x_{mN}+R_m(x_{mN}) \\
	&= A_{(m+1)N}A_{mN+N-1} \cdots A_{mN+1}y_{m}+R_m(y_{m}) 
	\end{align*}
	 where the remaining $R_m$ is a finite composition of the components of $\Omega_i$ and and $A_i$, $i=mN,\dots,(m+1)N$ and the required estimate on $R_m$ follows from the corresponding estimate on $\Omega_i$ and the fact that all the entries of the matrices $A_i$ are bounded. Finally, since $t^*_{n+N}-t^*_n\in\mathbb{N}$ for every $n\geq 0$ and $f$ is one periodic, $\ddot{f}(t^*_n) = \ddot{f}(t^*_{n+N})$ so that $A_n=A_{n+N}$ for every $n\geq 0$. This implies that for every $m\geq 0$ the matrix
	 \[
	 A_{(m+1)N}A_{mN+N-1} \cdots A_{mN+1}=A_{0}A_{N-1} \cdots A_{1}:=A
	 \]
	is independent on $m$. 
	\end{proof}

We are ready to apply Theorem \ref{cl} to the map $P_f^N$.

\begin{prop}\label{contu}
Suppose that
\[
|\Tr(A_{N-1}\dots A_1A_0 )|>2.
\]
with $A_n$ defined in Proposition \ref{proplin}. Then there exists a one dimensional continuum $\tilde{S}\ni(t_0^*,v_0^*)$ of initial data leading to unbounded solutions of the map $P_f$.
\end{prop}
\begin{proof}
With reference to Proposition \ref{linN}, we have that $\det A=1$ and
\[
|\Tr(A)| =|\Tr(A_{0}A_{N-1}\dots A_{1})| =|\Tr(A_{N-1}\dots A_1 A_{0})|>2
\]
using the fact that the trace of a product of matrices is invariant under cyclic permutations. Since $A$ is a $2\times 2$ matrix, it is hyperbolic and the eigenvalues $\lambda_1,\lambda_2$ satisfy $0<|\lambda_2|<1<|\lambda_2|$. Hence we can apply Theorem \ref{cl} to the map $P^N(\tau_0,\nu_0)$ and find for every $m\geq m_0$ large enough a one dimensional continuum $S_{m}$ such that if $(\tau_{m_0N},\nu_{m_0N})\in S_{m_0}$ then $(\tau_{mN},\nu_{mN})\in S_{m}$ and $(\tau_{mN},\nu_{mN})\to(0,0)$ as $m>m_0$ tends to $+\infty$. \\
Coming back to the variables $(t,v)$ we have that for every $m\geq m_0$  there exists a continuum $\tilde{S}_{m}$ such that if $(t_{m_0 N},v_{m_0 N})\in\tilde{S}_{m_0} $ then $(t_{m N},v_{m N})\in\tilde{S}_{m} $ and
\[
t_{mN}-t_{mN}^*\to 0,\qquad v_{mN}-v_{mN}^*\to 0 \quad\mbox{ as }m\to\infty.
\]  
In order to construct the continuum $\tilde{S}$ of initial conditions, we note that for every $n>0$, $P_f^{-n}(t_n,v_n)$ is well defined in a sufficiently small neighbourhood of the unbounded orbit $(t_n^*,v_n^*)$ and the corresponding $v$-component is larger that $\bar{v}$. \\
 Therefore, if
$(t_0,v_0)\in(P_f^N)^{-m_0}(\tilde{S}_{m_0}):=\tilde{S}$ then   
we get the thesis remembering the properties of $(t_n^*,v_n^*)$ in \eqref{unborb}.       
\end{proof}
We finally show that this applies to the bouncing motions corresponding to the family $p_s$ defined in Proposition \ref{oneparfam}, concluding the proof of Theorem \ref{main} and Corollary \ref{main2}. 
\begin{prop}
Consider the family of functions $p_s(t), s \in I$ defined in Proposition \ref{oneparfam}. Then there exists a real open interval $\tilde{I}\subset I$ such that for every $s\in \tilde{I}$ there exists a one dimensional continuum $\tilde{S}_s$ such that the orbits of $P_{p_s}$ starting in $\tilde{S}_s$ are unbounded. Moreover the continua $\tilde{S}_s$ intersect in the unbounded orbit found in Proposition \ref{oneparfam}.    
\end{prop}
\begin{proof}
	By Proposition \ref{oneparfam} there exists an unbounded orbit, independent on $s$ satisfying \eqref{unborb} for $N=2$ with $t_0^*=0$ and $t_1^*=5/12+k$. Hence, we only need to check that for every $s\in I$
	\begin{align*}
	T_s=|\Tr(A_1A_0)|&=\left|\Tr \left(\begin{array}{cc}
		1 & \frac{2}{g}  \\ 
		2\ddot{p}_s(t_{1}^*) & 1+\frac{4}{g}\ddot{p}_s(t_{1}^*)
	\end{array}
	\right) \left(\begin{array}{cc}
		1 & \frac{2}{g}  \\ 
		2\ddot{p}_s(t_{0}^*) & 1+\frac{4}{g}\ddot{p}_s(t_{0}^*)
	\end{array}
	\right)\right| \\
	&= \left|2+\frac{8}{g}\left(\ddot{p}_s(t_{0}^*)+\ddot{p}_s(t_{1}^*)+\frac{2}{g}\ddot{p}_s(t_{0}^*)\ddot{p}_s(t_{1}^*)\right)\right|>2
	\end{align*} 
We recall that from the proof of Proposition \ref{oneparfam}, the family $p_s(t)$ has the form, for $s\in I$ 
\[
	p_s(t) = a_1(s) \sin(2\pi t)+b_1(s) \cos(2\pi t) + a_2(s) \sin(4\pi t) + b_2(s) \cos(4\pi t),
\] 
where the coefficients are defined in \eqref{coeff} and $0.006\in I$. Moreover, $\ddot{p}_s(t)=g\tilde{p}_s(t)$ with $\tilde{p}_s(t)$ independent on $g$. Hence,
\[
T_s=\left|2+8\left(\tilde{p}_s(t_{0}^*)+\tilde{p}_s(t_{1}^*)+2\tilde{p}_s(t_{0}^*)\tilde{p}_s(t_{1}^*)\right)\right|
\] 
and from a direct computation we get that
\[
\tilde{p}_{0.006}(t_0^*)>0, \qquad \tilde{p}_{0.006}(t_1^*)>0.
\]
Finally, by continuity we can find a new interval $\tilde{I}\subset I$ such that for every $s\in \tilde{I}$
\[
\tilde{p}_{s}(t_0^*)>0, \qquad \tilde{p}_{s}(t_1^*)>0.
\]
Hence $T_s>2$ for $s\in \tilde{I}$ and we can conclude applying Proposition \ref{contu}.  	
	
\end{proof}

\begin{remark}
	With reference to Remark \ref{numeric}, it follows that $T_s>2$ also for the optimal numerical value 
	\[
	\bar{s}=0.009569094523943.
	\]
	In this case, 
	\[
\tilde{p}_s(t_{0}^*)+\tilde{p}_s(t_{1}^*)+2\tilde{p}_s(t_{0}^*)\tilde{p}_s(t_{1}^*)=1.186500669840734>0.
\]
\end{remark}

\vspace{1cm} 

\noindent \textbf{Acknowledgements.} This work is part of the author activity
within the DinAmicI community (www.dinamici.org) and the Gruppo Nazionale di Fisica Matematica, INdAM, Italy. The author is also grateful to the referees whose suggestions significantly improved the final version of the paper.


\end{document}